\documentclass[12pt,reqno]{amsart}
\usepackage{enumerate}
\usepackage{amsrefs}
\usepackage{amsfonts, amsmath, amssymb, amscd, amsthm, bm, cancel}
\usepackage{mathrsfs}
\usepackage{url}
\usepackage{graphicx}
\usepackage[
linktocpage=true,colorlinks,citecolor=magenta,linkcolor=blue,urlcolor=magenta]{hyperref}
\usepackage{appendix}
\usepackage{multicol}
\usepackage{comment}
\usepackage{tikz}
\usetikzlibrary{intersections}
\usepackage[margin=1in]{geometry}
\newtheorem{thm}{Theorem}[section]
\newtheorem*{thmA}{Theorem A}
\newtheorem*{thmB}{Theorem B}
\newtheorem*{thmC}{Theorem C}

\newtheorem{cor}{Corollary}[section]

\newtheorem{lem}{Lemma}[section]

\theoremstyle{definition}

\newtheorem*{ack}{Acknowledgments}

\theoremstyle{remark}
\newtheorem{rem}{Remark}[section]

\numberwithin{equation}{section}
\numberwithin{figure}{section}

\allowdisplaybreaks

\renewcommand{\(}{\left(}
\renewcommand{\)}{\right)}

\newcommand{\mrm}{\mathrm}

\newcommand{\Vol}{\mrm{Vol}}

\newcommand{\divv}{\mrm{div}}
\newcommand{\abs}[1]{\lvert#1\rvert}

\newcommand{\metric}[2]{\ensuremath{\langle #1, #2\rangle}}  

\begin{document}
\title{A class of weighted isoperimetric inequalities in hyperbolic space}

\author{Haizhong Li}
\address{Department of Mathematical Sciences, Tsinghua University, Beijing 100084, P.R. China}
\email{\href{mailto:lihz@tsinghua.edu.cn}{lihz@tsinghua.edu.cn}}

\author{Botong Xu}
\address{Department of Mathematical Sciences, Tsinghua University, Beijing 100084, P.R. China}
\email{\href{mailto:xbt17@mails.tsinghua.edu.cn}{xbt17@mails.tsinghua.edu.cn}}

\maketitle	
	\begin{abstract}
		In this paper, we prove a class of weighted isoperimetric inequalities for bounded domains in hyperbolic space by using the isoperimetric inequality with log-convex density in Euclidean space. 
		As a consequence, we remove the horo-convex assumption of domains in a weighted isoperimetric inequality proved by Scheuer-Xia. 
		Furthermore, we prove weighted isoperimetric inequalities for star-shaped domains in warped product manifolds.
		Particularly, we obtain a weighted isoperimetric inequality for star-shaped hypersurfaces lying outside a certain radial coordinate slice in the anti-de Sitter-Schwarzschild manifold.  
	\end{abstract}
	
	\maketitle
	\section{Introduction}\label{Sec-1}
	Research on weighted isoperimetric  inequalities for bounded domains in manifolds is popular in recent years.
	In Euclidean space $\mathbb{R}^{n+1}$, one of the most important results is the isoperimetric inequality with log-convex density $\phi(\abs{x})$ proved by Chambers \cite{Cha19}. A positive, $C^2$ function $\phi$ is called log-convex if $\left( \log \phi\right)'' \geq 0$. The log-convex assumption is reasonable in the sense that balls centered at the origin are stable under volume-preserving variations, see \cite{RCBM08} . For weighted isoperimetric inequalities in Euclidean space, we refer to Morgan's book \cite{Mor16}. 
	
	For horo-convex
	\footnote{A domain in the hyperbolic space is called horo-convex if all the principal curvatures on its boundary are bounded below by $1$.}
	domains in hyperbolic space $\mathbb{H}^{n+1}$, Scheuer-Xia \cite{SX19} proved the following weighted isoperimetric inequality. In the context, we let $n \geq 1$ be an integer and use $\omega_n$ to denote the area of the unit sphere $\mathbb{S}^n$.
	\begin{thmA}[\cite{SX19}]
		Let $\Omega$ be a smooth horo-convex bounded domain in $\mathbb{H}^{n+1}$ with the origin lying inside $\Omega$. Then  
		\begin{equation*}
		\(\int_{\partial \Omega} \cosh r d\mu\)^2 \geq    \( (n+1) \int_{\Omega} \cosh r dv \)^2 + \omega_n^{\frac{2}{n+1}}  \(  (n+1)\int_{\Omega} \cosh r dv \)^{\frac{2n}{n+1}}.
		\end{equation*}
		Equality holds if and only if $\Omega$ is a geodesic ball centered at the origin.
	\end{thmA}
	Scheuer-Xia \cite{SX19} used a locally constrained inverse curvature flow to prove a weighted Minkowski-type inequality for closed, star-shaped, mean convex hypersurfaces in $\mathbb{H}^{n+1}$.
	Combining that result with a Minkowski type inequality proved by Xia \cite{Xia16}, they proved Theorem A. Later, the first named author with Hu \cite{HL21} proved a sequence of geometric inequalities for weighted curvature integrals by using different locally constrained curvature flows. A special case of \cite[Theorem 1.7]{HL21} is the same as Theorem A
	\footnote{The first named author with Hu \cite{HL21} pointed out that Theorem A holds when $\Omega$ is a static convex domain containing the origin in its interior. The static convexity implies the strict convexity, but is weaker than horo-convexity.}
	.
	The main result of this paper is a class of weighted isoperimetric inequalities for bounded domains in $\mathbb{H}^{n+1}$. 
	\begin{thm}\label{thm-general weighted iso ineq}
		Let $\Omega$ be a bounded domain in $\mathbb{H}^{n+1}$ with $C^1$ boundary. Let $\phi(t)$ be a smooth, positive, even function on $\mathbb{R}$ that satisfies $\left(\log \phi\right)'' \geq 0$.  Then
		\begin{equation}\label{main ineq}
		\int_{\partial \Omega} \phi(\sinh r) \cosh r d\mu \geq \psi \left(\int_\Omega \phi(\sinh r) \cosh r dv  \right),
		\end{equation}
		where $\psi$ satisfies
		\begin{equation}\label{def-psi}
		\psi\left( \omega_n \int_0^t \phi(s) s^n ds \right) = \omega_n \phi(t)t^n \sqrt{t^2+1}, \quad \forall \ t \geq 0.
		\end{equation}
		Equality holds in \eqref{main ineq} if and only if $\Omega$ is a geodesic ball centered at the origin.
	\end{thm}
	Theorem \ref{thm-general weighted iso ineq} can be viewed as a counterpart of Chambers's result \cite{Cha19} in hyperbolic space. Furthermore, by letting  $\phi \equiv 1$ in Theorem \ref{thm-general weighted iso ineq}, we show that the bounded domain $\Omega$ in Theorem A need not be convex and need not contain the origin in its interior.
	\begin{cor}\label{cor-weighted-iso-ineq}
		Let $\Omega$ be a bounded domain in $\mathbb{H}^{n+1}$ with  $C^1$ boundary, then
		\begin{equation*}
		\(\int_{\partial \Omega} \cosh r d\mu\)^2 \geq   \( (n+1) \int_{\Omega} \cosh r dv \)^2 + \omega_n^{\frac{2}{n+1}}  \(  (n+1)\int_{\Omega} \cosh r dv \)^{\frac{2n}{n+1}}.
		\end{equation*}
		Equality holds if and only if $\Omega$ is a geodesic ball centered at the origin.
	\end{cor}
	
	Denote by $B(r)$  the geodesic ball of radius $r>0$ centered at the origin in $\mathbb{H}^{n+1}$. 
	Assume that $0 \leq k \leq n$ is an integer. Let  $\widetilde{W}_k(B(r))$ denote the $k$-th modified quermassintegral of $B(r)$ and $\tilde{f}_k(r) := \widetilde{W}_k(B(r))$. Let $\tilde{h}_k(r) := \omega_n e^{-(k+1)r} \sinh^{n-k} r$. For the definition of modified quermassintegrals, we refer readers to \cite{ACW18, HLW20}. We remark that in the case $k=0$, $\widetilde{W}_0 (\Omega) := \Vol(\Omega)$ and hence $\tilde{f}_0(r)= \omega_n \int_0^r \sinh^n t dt$. Let $\tilde{\kappa}_i := \kappa_i-1$ be the shifted principal curvatures of a hypersurface $\Sigma \subset \mathbb{H}^{n+1}$. For each integer $1 \leq m \leq n$, $E_m(\tilde{\kappa})$ is the normalized $m$-th elementary symmetric function of $\tilde{\kappa}$ and  $\Gamma^{+}_m$ is the Garding cone. In this paper, we use $u := \metric{\sinh r \partial_r}{\nu}$ to denote the support function of a hypersurface $\Sigma \subset \mathbb{H}^{n+1}$, where $r$ denotes the distance from points on $\Sigma$ to the origin, and $\nu$ is the unit outer normal of $\Sigma$.
	The first named author with Hu and Wei \cite{HLW20} proved the following weighted Alexandrov-Fenchel inequalities for horo-convex bounded domains. 
	\begin{thmB}[\cite{HLW20}]
		Let $\Sigma = \partial \Omega$ be a smooth, horo-convex hypersurface in $\mathbb{H}^{n+1}$ with the origin lying inside $\Omega$. Let $k=1,\cdots,n$. If $\tilde{\kappa} \in \Gamma^+_k$, then there holds
		\begin{equation*}
		\int_\Sigma \(\cosh r-u \) E_k (\tilde{\kappa}) d\mu \geq \tilde{h}_k \circ \tilde{f}_k^{-1} \( \widetilde{W}_k (\Omega)\).
		\end{equation*}
		Equality holds if and only if $\Omega$ is a geodesic ball centered at the origin.
	\end{thmB}
	As applications of Theorem \ref{thm-general weighted iso ineq}, we prove the following corollaries, which are also weighted isoperimetric inequalities.
	\begin{cor}\label{cor-k=0 case of HLW}
		Let $\Omega$ be a bounded domain in $\mathbb{H}^{n+1}$ with $C^1$ boundary. Then
		\begin{equation*}
		\int_{\partial \Omega} (\cosh r- u) d\mu \geq \tilde{h}_0 \circ \tilde{f}_0^{-1} (\Vol(\Omega)),
		\end{equation*}
		where $\tilde{h}_0(r) = \omega_n e^{-r} \sinh^{n} r$ and $\tilde{f}_0(r) = \omega_n \int_0^r \sinh^n t dt$.
		Equality holds if and only if $\Omega$ is a geodesic ball centered at the origin.
	\end{cor}
	\begin{rem}
		We note that $E_0(\tilde{\kappa}) \equiv 1$ on $\partial \Omega$ and $\widetilde{W}_0 (\Omega) := \Vol(\Omega)$.
		Thus, Corollary \ref{cor-k=0 case of HLW} corresponds to the remaining $k=0$ case in Theorem B.
	\end{rem}
	
	The following Corollary \ref{cor-another app} corresponds to the remaining $k=0$ case in \cite[Theorem 1.4]{HLW20} by a similar argument as above. Hence we omit the details.
	\begin{cor}\label{cor-another app}
		Let $\Omega$ be a bounded domain in $\mathbb{H}^{n+1}$ with  $C^1$ boundary. Then
		\begin{equation*}
		\int_{\partial \Omega} \cosh r d\mu \geq h_0\circ f_0^{-1} (\Vol(\Omega)),
		\end{equation*}
		where $h_0(r) =\omega_n \cosh r \sinh^n r$ and $f_0(r)=\omega_n \int_0^r \sinh^n t dt$.
		Equality holds if and only if $\Omega$ is a geodesic ball centered at the origin.
	\end{cor}
	
	Let $M^{n+1}= [a, b] \times \mathbb{S}^n$ be a warped product manifold with metric 
	\begin{equation*}
	\bar{g} = dr^2 + \lambda^2(r) g_{\mathbb{S}^n},
	\end{equation*}
	where $\lambda(a) \geq 0$, $\lambda' \geq 1$ on $[a,b]$, and $b$ could be infinity. 
	Define $\Psi(t) := \lambda' \circ \lambda^{-1} (t) $ and $\Lambda (t) := \frac{\sqrt {\Psi^2(t) -1 }}{t}$ for $ t \in [\lambda(a), \lambda(b)]$.
	Using a similar method as in the proof of Theorem \ref{thm-general weighted iso ineq}, we prove a class of weighted isoperimetric inequalities for star-shaped domains in warped product manifolds.
	\begin{thm}\label{thm-iso-warped}
		Let $\Sigma$ be a star-shaped, $C^1$ hypersurface in $M^{n+1}$. Let $\Omega$ be the domain bounded by $\Sigma$ and $\{a\} \times \mathbb{S}^n$. 
		Assume that $\Lambda(t)$ is smooth, non-decreasing and satisfies $\( \Lambda'(t) t \)' \geq 0$ for all $ t \in [\lambda(a), \lambda(b)]$.
		Let $\phi(t)$ be a smooth, positive, even function on $[-\lambda(b), \lambda(b)]$ that satisfies $\(\log \phi\)'' \geq 0$.  Then
		\begin{equation}\label{weighted isope warped produc}
		\int_\Sigma \phi(\lambda(r)) \lambda'(r) d\mu \geq \tilde{\psi} \( \int_\Omega \phi(\lambda(r)) \lambda'(r) dv\), 
		\end{equation}
		where $\tilde{\psi}$ satisfies
		\begin{equation*}
		\tilde{\psi} \( \omega_n \int_{\lambda(a)}^t \phi(s) s^n ds\) = \omega_n \phi(t) t^n \Psi(t), \quad \forall \ t \in [\lambda(a), \lambda(b)].
		\end{equation*}
		Equality holds in \eqref{weighted isope warped produc} if and only if $\Sigma$ is a radial coordinate slice $\{r_0\} \times \mathbb{S}^n$, where $r_0 \in [a, b]$. 
	\end{thm}
	\begin{rem}
		In the case $M^{n+1} = \mathbb{H}^{n+1}$, we have $\lambda(r) = \sinh r$, $\Psi(t) = \cosh \( \sinh^{-1}(t) \) = \sqrt{t^2+1}$ and $\Lambda(t) =1$ for all $t \in [0, +\infty)$. Then Theorem \ref{thm-iso-warped} reduces to Theorem \ref{thm-general weighted iso ineq} when $\Sigma= \partial \Omega$ is star-shaped.
	\end{rem}
	
	The anti-de Sitter-Schwarzschild manifold can be viewed as a special type of warped product manifolds, where $\lambda(r)$ satisfies
	\begin{equation}\label{warp-func-ads-S}
	\lambda'(r) = \sqrt{1+ \lambda^2- m \lambda^{1-n}}
	\end{equation}
	for some constant $m \geq 0$.
	As a consequence, $\lambda''(r) = \lambda(r) + \frac{m(n-1)}{2}\lambda^{-n} $.  We refer readers to \cite{BHW16} for more details of the anti-de Sitter-Schwarzschild manifold.
	If we remove the restrictions of $\Lambda (t)$ in Theorem \ref{thm-iso-warped}, then we can also get interesting inequalities. For example, we obtain a weighted isoperimetric inequality for  star-shaped hypersurfaces in the anti-de Sitter-Schwarzschild manifold that lies outside a certain radial coordinate slice.
	\begin{thm}\label{thm-iso-ads-S}
		Let $M^{n+1} = [a,+ \infty) \times \mathbb{S}^n$ be equipped with the anti-de Sitter-Schwarzschild metric,  where $a$ satisfies $\lambda(a) =m^{\frac{1}{n+1}}$.
		Let $\Sigma$ be a star-shaped, $C^1$ hypersurface in $M^{n+1}$. Let $\Omega$ be the domain bounded by $\Sigma$ and $\{a\} \times \mathbb{S}^n$. Then
		\begin{equation}\label{weighted isope adsS}
		\begin{aligned}
		\(\int_\Sigma \lambda'(r) d\mu\)^2 
		\geq& 
		\(\(n+1\)\widehat{\eta}\(\int_\Omega \lambda'(r) dv  \)
		+ \frac{m(n+1)}{2}\int_{\Omega} \frac{\lambda'(r)}{\lambda^n(r) \sqrt{\( \lambda'(r) \)^2-1}}dv\)^2\\
		&+ \( (n+1) \omega_n^{\frac{1}{n}}  \int_\Omega \lambda'(r) dv+ \omega_n^{\frac{n+1}{n}} m  \)^\frac{2n}{n+1},
		\end{aligned}
		\end{equation}
		where $\widehat{\eta}$ satisfies 
		\begin{equation}\label{def-eta-hat}
		\widehat{\eta} \(\omega_n \int_{m^{\frac{1}{n+1}}}^{t} s^n ds \) = \omega_n \int_{m^{\frac{1}{n+1}}}^t \Lambda(s) s^n ds, \quad  \forall \ t \geq m^{\frac{1}{n+1}}. 
		\end{equation}
		Equality holds in \eqref{weighted isope adsS} if and only if $\Sigma$ is a radial coordinate slice $\{r_0\} \times \mathbb{S}^n$, where $r_0 \in [a, +\infty)$.
	\end{thm}

	Since a large amount of work has been done for weighted isoperimetric inequalities, we list some of them for readers interested in this topic, see e.g. \cite{BDKS18, BBCLT16, BHW16, GWW15, GR20, Kwo16, Mor05}. 
	
	\begin{ack}
	The authors were supported by NSFC Grant No. 11831005 and NSFC Grant No. 12126405. This paper was submitted to Proc. Amer. Math. Soc. in March, 2022 and appears in Proc. AMS, see doi: 10.1090/proc/16219. We note that L. Silini studied the related problem recently in the paper ``Approaching the isoperimetric problem in $H^m_{\mathbb{C}}$ via the hyperbolic log-convex density conjecture",
	  arXiv: 2208.00195.
	\end{ack}

	\section{Weighted isoperimetric inequalities in $\mathbb{H}^{n+1}$}\label{Sec-2}
	In Section \ref{Sec-2}, we give the proofs of Theorem \ref{thm-general weighted iso ineq} and Corollaries \ref{cor-k=0 case of HLW} -- \ref{cor-another app}.
	
	The Minkowski space $\mathbb{R}^{n+1,1}$ is an $(n+2)$-dimensional vector space with metric 
	\begin{equation*}
	\metric{X}{Y} =\metric{(x_0,x_1, \cdots,x_{n+1})}{(y_0,y_1,\cdots, y_{n+1})}
	= \sum_{i=0}^n x_iy_i- x_{n+1}y_{n+1}.
	\end{equation*}
	We write $X = (x, x_{n+1})$ for $X \in \mathbb{R}^{n+1,1}$, where $x \in \mathbb{R}^{n+1}$. The hyperboloid model of the hyperbolic space is $\mathbb{H}^{n+1} = \lbrace X \in \mathbb{R}^{n+1,1} \left\vert  \metric{X}{X} = -1, x_{n+1}>0 \right. \rbrace$. That implies that $X$ is the normal vector of $\mathbb{H}^{n+1} \subset \mathbb{R}^{n+1,1}$. For any $X \in \mathbb{H}^{n+1}$, the polar coordinates of $\mathbb{H}^{n+1}$ give that $X = (x, x_{n+1}) = \(  \sinh r \theta, \cosh r\)$ for some $\theta \in \mathbb{S}^n$ and $r \geq 0$. Here $r$ represents the geodesic distance from $X$ to $(0,1)$ in $\mathbb{H}^{n+1}$. 
	In this section, we let $\lambda = \abs{x} = \sinh r$ and $\lambda' =x_{n+1} = \cosh r$.
	For the above facts, we refer readers to Petersen's book \cite{Pet16}.
	
	Now we define the projection map $\pi : \mathbb{H}^{n+1} \to \mathbb{R}^{n+1}$ by  $\pi  \(X \) = x$. 
	Let $\Sigma = \partial \Omega$ be a closed, $C^1$ hypersurface in $\mathbb{H}^{n+1}$. We define $\widehat{\Omega} = \pi(\Omega)$ and  $\widehat{\Sigma} = \pi (\Sigma)$. We use $d \widehat{v}$ to denote the volume element of $\mathbb{R}^{n+1}$. Obviously, $\widehat{\Sigma} \subset \mathbb{R}^{n+1}$ is a closed, $C^1$ hypersurface as $\pi$ is a diffeomorphism from $\mathbb{H}^{n+1}$ to $\mathbb{R}^{n+1}$. Let $\widehat{\nu}$ denote the unit outer normal of $\widehat{\Sigma}$ and $\widehat{u}:= \metric{x}{\widehat{\nu}}$ denote the Euclidean support function of $\widehat{\Sigma}$. For any point $x$ in $\mathbb{R}^{n+1}$, we let $\rho = \abs{x}$.
	
	\begin{lem}\label{lem-rel-weighted vol-Ecul vol}
		Let $\phi$ be any continuous function on $\mathbb{R}$.
		The weighted volume of a bounded domain $\Omega \subset \mathbb{H}^{n+1}$ equals to the weighted volume of $\widehat{\Omega} \subset \mathbb{R}^{n+1}$, i.e.
		\begin{equation}\label{rel-weighted vol-Ecul vol}
		\int_{\Omega} \phi(\lambda)\lambda' dv = \int_{\widehat{\Omega}} \phi(\rho) d \widehat{v}.
		\end{equation} 
	\end{lem}
	
	\begin{proof}
		It is easy to see that $dv= \lambda^n drd\sigma$ and $d\widehat{v} = \rho^n d\rho d\sigma$, where $d\sigma$ is the area element of unite sphere $\mathbb{S}^n$.
		Since $\rho =\abs{\pi (X)} =\lambda(r)$, we have
		\begin{equation}\label{rel-weighted vol-Euc vol}
		\phi(\lambda)\lambda' dv = \phi(\lambda)\lambda' \lambda^n dr d\sigma = \phi(\lambda)\lambda^n d\lambda d\sigma
		= \phi(\rho)\rho^n d \rho d\sigma =\phi(\rho) d \widehat{v}.
		\end{equation}
		Then we get Lemma \ref{lem-rel-weighted vol-Ecul vol} by integrating both sides of \eqref{rel-weighted vol-Euc vol} on $\Omega$ and $\widehat{\Omega}$ respectively.
	\end{proof}
	
	\begin{lem}\label{lem-nu-uhat}
		Let $\Sigma$ be a closed, $C^1$ hypersurface in $\mathbb{H}^{n+1}$.
		A unit normal vector of $\Sigma$ at $X\in \Sigma$ can be given as
		\begin{equation}\label{nu-uhat}
		\nu=\frac{(\widehat{\nu} + \widehat{u}x, \lambda' \widehat{u})}{\sqrt{\widehat{u}^2+1}},
		\end{equation}	
		where $\widehat{u} =\metric{x}{\widehat{\nu}}$ is the Euclidean support function of $\widehat{\Sigma} = \pi (\Sigma)$ at $x = \pi (X)$.
	\end{lem}
	
	\begin{proof}
		We only need to check the following facts: $\metric{X}{\nu} = 0$, $\metric{\nu}{\nu} =1$ and $\metric{\pi^{-1}_* e}{\nu} = 0$ for any vector $e \in T_x \widehat{\Sigma}$. Here $\pi^{-1}_* e$ denotes the pushforward vector of $e$ via map $\pi^{-1}$. The first two facts imply that $\nu$ is a unit vector in $T_X \mathbb{H}^{n+1}$. Then, the third fact confirms that  $\nu$ is a normal vector at $X \in \Sigma$.
		
		By a direct calculation, we have
		\begin{equation}\label{X,nu=0}
		\metric{X}{\nu} =\frac{ \metric{(x, \lambda')}{(\widehat{\nu} + \widehat{u}x, \lambda' \widehat{u}) }}{\sqrt{\widehat{u}^2+1}}
		=\frac{\metric{x}{\widehat{\nu}}+ \widehat{u}\abs{x}^2 -\(\lambda'\)^2 \widehat{u} }{\sqrt{\widehat{u}^2+1}} =0,
		\end{equation}
		where we used $\abs{x}^2 = \lambda^2$ and $\( \lambda'\)^2- \lambda^2=1$. Similarly, we find that 
		\begin{equation}\label{nu,nu=1}
		\metric{\nu}{\nu} = \frac{ \abs{\widehat{\nu}+ \widehat{u}x}^2- \( \lambda' \widehat{u}\)^2 }{\widehat{u}^2+1} 
		=\frac{1 + 2 \widehat{u}^2 + \lambda^2 \widehat{u}^2 - \(\lambda'\)^2 \widehat{u}^2}{\widehat{u}^2+1}
		=1.
		\end{equation}
		Since $\pi \((x, x_{n+1}) \) = x$,
		for any $e \in T_x \widehat{\Sigma}$, we can let $\pi^{-1}_* e = \( e, \alpha\)$ for some real number $\alpha$. As $\pi^{-1}_* e \in T_X \mathbb{H}^{n+1}$, we know $\metric{X}{\( e, \alpha\)} = 0$. Hence $\alpha \lambda' = \metric{x}{e}$. Therefore,  we get
		\begin{equation*}
		\pi^{-1}_* e = \(e, \frac{\metric{x}{e}}{\lambda'}\).
		\end{equation*} 
		This together with the assumption $\metric{e}{\widehat{\nu}}=0$ gives
		\begin{equation}\label{pi-1 e-nu}
		\metric{\pi^{-1}_* e}{\nu} =
		\frac{\metric{e}{\widehat{\nu}+ \widehat{u}x } -\widehat{u}\metric{x}{e}  }{\sqrt{\widehat{u}^2+1}}
		= \frac{\metric{e}{\widehat{\nu}}+\widehat{u}\metric{x}{e}- \widehat{u}\metric{x}{e}   }{\sqrt{\widehat{u}^2+1}} =0.
		\end{equation}
		Then Lemma \ref{lem-nu-uhat} follows from \eqref{X,nu=0},  \eqref{nu,nu=1} and \eqref{pi-1 e-nu}.
	\end{proof}
	
	In the next lemma, we transform the weighted area of $\Sigma$ to an integral on $\widehat{\Sigma}$. We use $\lrcorner$ to denote the interior multiplication. We refer to Lee's book \cite{Lee03} for common properties of the interior multiplication.
	\begin{lem}\label{lem-rel-weighted area-integral in Eucl}
		Let $\Sigma$ be a closed, $C^1$ hypersurface in $\mathbb{H}^{n+1}$.
		Let $\phi$ be any continuous function on $\mathbb{R}$.
		It holds that
		\begin{equation}\label{rel-weighted area-integral in Eucl}
		\int_\Sigma \phi(\lambda) \lambda' d\mu = \int_{\widehat{\Sigma}} \phi(\rho) \sqrt{\widehat{u}^2 +1} d\widehat{\mu}.
		\end{equation}
	\end{lem}
	\begin{proof}
		Using \eqref{nu-uhat}, we have 
		\begin{equation}\label{pi-nu nu}
		\metric{\pi_* \nu}{\widehat{\nu}} = \frac{\metric{\widehat{\nu}+ \widehat{u} x}{\widehat{\nu}}}{\sqrt{\widehat{u}^2+1}} = \sqrt{\widehat{u}^2+1}.
		\end{equation}
		Using \eqref{rel-weighted vol-Euc vol} and \eqref{pi-nu nu}, we obtain
		\begin{equation}\label{loc-rel-weighted area-integral in Eucl}
		\begin{aligned}
		\phi(\lambda) \lambda' d\mu=&  \nu \lrcorner\(\phi(\lambda)\lambda' dv\)\\
		=&  \(\pi_* \nu\) \lrcorner \(\phi(\rho)d \widehat{v} \)\\
		=&\phi(\rho) \metric{\pi_* \nu }{\widehat{\nu}} d \widehat{\mu}\\
		=& \phi(\rho) \sqrt{\widehat{u}^2+1} d\widehat{\mu}.
		\end{aligned}
		\end{equation}
		Then Lemma \ref{lem-rel-weighted area-integral in Eucl} follows by integrating both sides of \eqref{loc-rel-weighted area-integral in Eucl} on $\Sigma$ and $\widehat{\Sigma}$ respectively.
	\end{proof}
	
	\begin{lem}\label{lem-sym of phi}
		Let $\Omega$ be a bounded domain in $\mathbb{H}^{n+1}$. Let $\phi(t)$ be a function that satisfies the  assumptions in Theorem \ref{thm-general weighted iso ineq}. Then 
		\begin{equation}\label{sym of phi}
		\int_\Omega \phi'(\lambda) \lambda' \lambda dv \geq \eta \(\int_\Omega \phi(\lambda) \lambda' dv \),
		\end{equation}
		where 
		\begin{equation}\label{def-eta}
		\eta \( \omega_n \int_0^t \phi(s) s^n ds \) = \omega_n \int_0^t \phi'(s)s^{n+1} ds, \quad \forall \ t \geq 0.
		\end{equation}
		If $\phi(\lambda(r))$ is not constant on $\Omega$, then equality holds in \eqref{sym of phi} if and only if $\Omega$ is a geodesic ball centered at the origin. 
	\end{lem}
	
	\begin{proof}
		We set $dm = \phi(\lambda) \lambda' dv$. Recall that $B(r)$ denotes the geodesic ball of radius $r$ centered at the origin in $\mathbb{H}^{n+1}$.
		Let $r_0$ satisfy $\int_{B(r_0)} dm =\int_\Omega dm$. 
		Hence,
		\begin{equation}\label{assmp-r_0}
		\int_{\Omega / B(r_0)} dm = \int_{ B(r_0)/\Omega } dm.
		\end{equation}
		Using \eqref{def-eta}, we have
		\begin{equation}\label{eta for ball}
		\begin{aligned}
		\eta\(\int_\Omega \phi(\lambda) \lambda' dv\) 
		=&\eta\(\int_{B(r_0)} dm\)\\
		=&\eta \( \omega_n \int_0^{\lambda(r_0)} \phi(\lambda) \lambda^n d\lambda \)\\
		=&\omega_n \int_0^{\lambda{(r_0)}} \phi'(\lambda) \lambda^{n+1} d\lambda\\
		=&\int_{B(r_0)}  \frac{\phi'(\lambda) \lambda}{\phi(\lambda)} dm.
		\end{aligned}
		\end{equation}
		Since $\phi$ is even on $\mathbb{R}$, we have $\phi'(0) = 0$. Then the assumption $\(\log \phi\)'' \geq 0$ directly implies that $\(\log \phi\)' \geq 0$ on $[0, +\infty)$. 
		Then,
		\begin{equation}\label{mono of phi}
		\(\frac{\phi'(t) t}{\phi(t)}\)'=\( \log \phi \)''(t) t + (\log \phi)'(t) \geq 0.
		\end{equation}
		Now we can apply the symmetrization method to get
		\begin{equation} \label{ineq-sym}
		\begin{aligned}
		\int_\Omega \phi'(\lambda) \lambda' \lambda dv=&
		\int_\Omega \frac{\phi'(\lambda) \lambda}{\phi(\lambda)} dm \\
		=& \int_{\Omega \cap  B (r_0)} \frac{\phi'(\lambda) \lambda}{\phi(\lambda)} dm  + \int_{\Omega / B(r_0)} \frac{\phi'(\lambda) \lambda}{\phi(\lambda)} dm \\
		\geq& \int_{\Omega \cap  B (r_0)} \frac{\phi'(\lambda) \lambda}{\phi(\lambda)} dm + \frac
		{\phi'(\lambda(r_0)) \lambda(r_0)}{\phi(\lambda(r_0))}\int_{\Omega / B(r_0)} dm \\
		=& \int_{ B (r_0) \cap \Omega} \frac{\phi'(\lambda) \lambda}{\phi(\lambda)} dm + \frac
		{\phi'(\lambda(r_0))\lambda(r_0)}{\phi(r_0)}\int_{ B(r_0)/\Omega } dm \\
		\geq& \int_{B(r_0)}  \frac{\phi'(\lambda) \lambda}{\phi(\lambda)} dm \\
		=&\eta \(\int_\Omega \phi(\lambda) \lambda' dv \), 
		\end{aligned}
		\end{equation}
		where we used \eqref{mono of phi}, \eqref{assmp-r_0} and \eqref{eta for ball}. Thus we  proved \eqref{sym of phi}.
		If $\phi(\lambda)$ is not constant in $\Omega$, then we know that \eqref{mono of phi} is strict on some interval $[r_1, r_2]$ such that $0 \leq r_1 <r_2$ and $\int_{\Omega / B(r_2)} dm >0$. In the case that equality holds in \eqref{sym of phi}, we first assume that $\Omega$ is not $B(r_0)$. However, if equality holds in \eqref{ineq-sym}, then $r_2 \leq r_0$ and $r_1 \geq r_0$. That is a contradiction. Hence, in the case that $\phi(\lambda)$ is not constant in $\Omega$, equality holds in \eqref{sym of phi} if and only if $\Omega$ is a geodesic ball centered at the origin. 
		We complete the proof of Lemma \ref{lem-sym of phi}.
	\end{proof}
	
	\begin{rem}
		The key step in the proof of Lemma \ref{sym of phi} is  \eqref{mono of phi}. Similar arguments can be applied to many other comparisons of volumes with different weights. 
	\end{rem}
	
	Chambers \cite{Cha19} proved the following isoperimetric inequality with log-convex density. For convenience, we will use the settings and notations in this paper.
	\begin{thmC}[\cite{Cha19}]
		Let $\widehat{\Omega}$ be a bounded domain in $\mathbb{R}^{n+1}$ with $C^1$ boundary $\widehat{\Sigma}$.
		Let $\phi$ be a smooth, positive, even function on $\mathbb{R}$ that satisfies $\(\log \phi\)'' \geq 0$.
		Then
		\begin{equation*} 
		\int_{\widehat{\Sigma}} \phi(\rho) d\widehat{\mu} \geq \xi \( \int_{\widehat{\Omega}} \phi(\rho) d\widehat{v}\),
		\end{equation*}
		where $\rho$ denotes the radial function of $\widehat{\Sigma}$, and $\xi$ satisfies
		\begin{equation}\label{def-xi}
		\xi \( \omega_n \int_0^t \phi(s) s^n ds\) = \omega_n \phi(t) t^n, \quad \forall \ t \geq 0.
		\end{equation}
	\end{thmC}
	
	Now we give the proof of Theorem \ref{thm-general weighted iso ineq}.
	\begin{proof}[Proof of Theorem \ref{thm-general weighted iso ineq}]
		It is simple to see that $\chi(t) = \sqrt{t^2+1}$ is strictly convex on $\mathbb{R}$. Then Jensen's inequality for integrals implies that
		\begin{equation}\label{Jensen ineq}
		\frac{\int_{\widehat{\Sigma}} \phi(\rho)\sqrt{\widehat{u}^2+1} d\widehat{\mu}}{\int_{\widehat{\Sigma}} \phi(\rho) d \widehat{\mu}}
		\geq \sqrt{  \(\frac{\int_{\widehat{\Sigma}} \phi(\rho)\widehat{u} d\widehat{\mu}}{\int_{\widehat{\Sigma}} \phi(\rho) d \widehat{\mu}}\)^2+1  }.
		\end{equation}
		By the divergence theorem, we know that 
		\begin{equation*}
		\int_{\widehat{\Sigma}} \phi(\rho) \widehat{u} d\widehat{\mu}
		= \int_{\widehat{\Omega}} \divv\( \phi(\rho) x \) d\widehat{v}
		= (n+1)\int_{\widehat{\Omega}} \phi(\rho) d \widehat{v} + \int_{\widehat{\Omega}} \phi'(\rho) \rho d \widehat{v}.
		\end{equation*}
		Then we can simplify \eqref{Jensen ineq} to get
		\begin{equation}\label{last ineq Eucl space}
		\begin{aligned}
		\(\int_{\widehat{\Sigma}} \phi(\rho)\sqrt{\widehat{u}^2+1} d\widehat{\mu}\)^2
		\geq& \(  (n+1)\int_{\widehat{\Omega}} \phi(\rho) d \widehat{v} + \int_{\widehat{\Omega}} \phi'(\rho) \rho d \widehat{v} \)^2 +\(\int_{\widehat{\Sigma}} \phi(\rho) d \widehat{\mu} \)^2 \\
		\geq& \(  (n+1)\int_{\widehat{\Omega}} \phi(\rho) d \widehat{v} + \int_{\widehat{\Omega}} \phi'(\rho) \rho d \widehat{v} \)^2 +\xi\(\int_{\widehat{\Omega}} \phi(\rho) d \widehat{v} \)^2, 
		\end{aligned}
		\end{equation}
		where we used Theorem C in the second inequality. 
		By using \eqref{rel-weighted area-integral in Eucl} and \eqref{rel-weighted vol-Ecul vol}  in \eqref{last ineq Eucl space}, we obtain
		\begin{equation}\label{iso-eta-xi}
		\begin{aligned}
		\(\int_{\Sigma} \phi(\lambda) \lambda' d\mu\)^2 \geq&
		\(  (n+1) \int_{\Omega}\phi(\lambda) \lambda' dv + \int_{\Omega} \phi'(\lambda) \lambda' \lambda dv \)^2 + \xi \(\int_\Omega \phi(\lambda) \lambda'dv\)^2  \\
		\geq&
		\(  (n+1) \int_{\Omega}\phi(\lambda) \lambda' dv + \eta\(\int_{\Omega} \phi(\lambda) \lambda'  dv\)  \)^2 + \xi \(\int_\Omega \phi(\lambda) \lambda'dv\)^2,
		\end{aligned}
		\end{equation}
		where we used \eqref{sym of phi} in the second inequality.
		Let $t_0$ satisfy $\int_{\Omega} \phi(\lambda) \lambda' dv = \omega_n \int_0^{t_0} \phi(s) s^n ds$. Using \eqref{def-eta} and \eqref{def-xi} in \eqref{iso-eta-xi}, we have
		\begin{equation*}
		\begin{aligned}
		\(\int_{\Sigma} \phi(\lambda) \lambda' d\mu\)^2 \geq&
		\(\omega_n \(\int_0^{t_0} (n+1)\phi(s)s^n ds\) + \omega_n \int_0^{t_0} \phi'(s)s^{n+1} ds \)^2 + \( \omega_n \phi(t_0) t_0^n \)^2 \\
		=&  \( \omega_n \phi(t_0) t_0^{n+1}\)^2 + \(\omega_n \phi(t_0)t_0^n \)^2 \\
		=& \(\omega_n \phi(t_0) t_0^n\sqrt{t_0^2+1} \)^2\\
		=& \psi \(\omega_n \int_0^{t_0} \phi(s)s^n ds\)^2 \\
		=& \psi \(  \int_{\Omega} \phi(\lambda) \lambda' dv\)^2,
		\end{aligned}
		\end{equation*}
		where we used \eqref{def-psi} in the third equality. Thus we proved \eqref{main ineq}.
		However, equality holds in \eqref{Jensen ineq} if and only if $\widehat{u}$ is constant on $\widehat{\Sigma}$. At the critical points of $|x|^2$ on $\widehat{\Sigma}$, we have $\abs{x} = \abs{\widehat{u}}$. Hence $\abs{x}$ is constant on $\widehat{\Sigma}$, which implies that $\widehat{\Omega}$ is a geodesic ball centered at the origin. 
		Hence, if equality holds in \eqref{main ineq}, then $\widehat{\Omega} \subset \mathbb{R}^{n+1}$ is a geodesic ball centered at the origin. In this case, $\Omega \subset \mathbb{H}^{n+1}$ is a geodesic ball centered at the origin. Conversely, it is easy to show that equality holds in \eqref{main ineq} when $\Omega$ is a geodesic ball centered at the origin.
		We complete the proof of Theorem \ref{thm-general weighted iso ineq}.
	\end{proof}
	
	In the following Lemma \ref{lem-ineq-weighted vol-cano col}, we study the comparison between the weighted volume and the canonical volume of domains in $\mathbb{H}^{n+1}$. Although the proof is similar to that of Lemma \ref{lem-sym of phi}, it is interesting in its own right. For convenience, we set $\Vol_w(\Omega) : =\int_{\Omega} \lambda' dv$ and $r_\Omega := \tilde{f}_0^{-1} \(\Vol(\Omega)\)$ for a domain $\Omega \subset \mathbb{H}^{n+1}$. Recall that $\tilde{f}_0$ was defined in Section \ref{Sec-1}.
	\begin{lem}\label{lem-ineq-weighted vol-cano col}
		Let $\Omega$ be a bounded domain in $\mathbb{H}^{n+1}$, then 
		\begin{equation*}
		\int_{\Omega} \lambda' dv \geq \frac{\omega_n}{n+1} \lambda^{n+1} (r_\Omega).
		\end{equation*}
		Equality holds if and only if $\Omega$ is a geodesic ball centered at the origin.
	\end{lem}
	\begin{proof}
		It is easy to see that $\Vol \(\Omega / B (r_\Omega) \) = \Vol \(B(r_\Omega) /\Omega \)$ by the definition of $r_\Omega$. Since $\lambda'(r) = \cosh r$ is strictly increasing in $r$ for $r>0$, we have
		\begin{equation*}
		\begin{aligned}
		\Vol_w(\Omega) =& \Vol_w \(\Omega \cap  B (r_\Omega) \) + \Vol_w \(\Omega / B(r_\Omega) \)\\
		\geq& \Vol_w \(\Omega \cap  B (r_\Omega) \)+ \lambda' (r_\Omega) \Vol\( \Omega / B(r_\Omega) \)\\
		\geq& \Vol_w \( B(r_\Omega)\)\\
		=& \frac{\omega_n}{n+1} \lambda^{n+1} (r_\Omega) .
		\end{aligned}
		\end{equation*}	
		Equality holds if and only if $\Omega$ is a geodesic ball centered at the origin.
	\end{proof}
	
	We are now in a position to prove Corollary \ref{cor-k=0 case of HLW} and Corollary \ref{cor-another app}. Recall that Corollary \ref{cor-weighted-iso-ineq} follows from Theorem \ref{thm-general weighted iso ineq} directly. 
	\begin{proof}[Proof of Corollary \ref{cor-k=0 case of HLW}]
		Using Corollary \ref{cor-weighted-iso-ineq} and the divergence theorem, we have
		\begin{equation}\label{poly-Vw}
		\begin{aligned}
		&\int_{\partial \Omega} (\cosh r- u) d\mu
		= \int_{\partial \Omega} \cosh r d\mu - (n+1) \Vol_w(\Omega) \\
		\geq& 
		\( \( (n+1)\Vol_w(\Omega) \)^2+
		\omega_n^{ \frac{2}{n+1}} \( (n+1)\Vol_w(\Omega) \)^{\frac{2n}{n+1}}
		\)^{\frac{1}{2}} - (n+1) \Vol_w(\Omega).
		\end{aligned}
		\end{equation}
		It is easy to see that the right hand side of \eqref{poly-Vw} is monotone increasing as a single-variable function of  $\Vol_w(\Omega)$. Therefore, by using  Lemma \ref{lem-ineq-weighted vol-cano col}, we have
		\begin{equation*}
		\begin{aligned}
		\int_{\partial \Omega} (\cosh r -u) d\mu
		\geq& \(  \(\omega_n \sinh^{n+1} (r_\Omega) \)^2 + \omega_n^{\frac{2}{n+1} } \(\omega_n \sinh^{n+1}( r_\Omega) \)^{\frac{2n}{n+1}}    \)^{\frac{1}{2}} -\omega_n \sinh^{n+1}( r_\Omega)\\
		=& \omega_n  e^{-r_\Omega} \sinh^n (r_\Omega) \\
		=& \tilde{h}_0 \circ \tilde{f}_0^{-1} (\Vol(\Omega)).
		\end{aligned}
		\end{equation*}
		Equality holds if and only if $\Omega$ is a geodesic ball centered at the origin. We complete the proof of Corollary \ref{cor-k=0 case of HLW}.
	\end{proof}
	
	\begin{proof}[Proof of Corollary \ref{cor-another app}]
		Since $f_0 = \tilde{f}_0$, we have $r_\Omega = \tilde{f}_0^{-1} \( \Vol(\Omega) \) = f_0^{-1}\( \Vol(\Omega) \)$.
		Using Corollary \ref{cor-weighted-iso-ineq} and Lemma \ref{lem-ineq-weighted vol-cano col}, we have
		\begin{equation*}
		\begin{aligned}
		\int_{\partial \Omega} \cosh r d\mu
		\geq&
		\( \( (n+1)\Vol_w(\Omega) \)^2+
		\omega_n^{ \frac{2}{n+1}} \( (n+1)\Vol_w(\Omega) \)^{\frac{2n}{n+1}}
		\)^{\frac{1}{2}}\\
		\geq&
		\(  \(\omega_n \sinh^{n+1} (r_\Omega) \)^2 + \omega_n^{\frac{2}{n+1} } \(\omega_n \sinh^{n+1}( r_\Omega) \)^{\frac{2n}{n+1}}    \)^{\frac{1}{2}}\\
		=&\omega_n \cosh\(r_\Omega\) \sinh^n \(r_\Omega\)\\
		=& h_0 \circ f_0^{-1}\(\Vol (\Omega)\).
		\end{aligned}
		\end{equation*}
		Equality holds if and only if $\Omega$ is a geodesic ball centered at the origin. We complete the proof of Corollary \ref{cor-another app}.
	\end{proof}
	
	\section{Weighted isoperimetric inequalities in warped product manifolds} \label{Sec-3}
	In Section \ref{Sec-3}, we give the proofs of Theorem \ref{thm-iso-warped} and Theorem \ref{thm-iso-ads-S}.
	
	Let $M^{n+1}= [a, b] \times \mathbb{S}^n$ be a warped product manifold with metric 
	\begin{equation*}
	\bar{g} = dr^2 + \lambda^2(r) g_{\mathbb{S}^n},
	\end{equation*}
	where $\lambda(a) \geq 0$, $\lambda' \geq 1$ on $[a,b]$, and $b$ could be infinity.
	Then the volume element of $M^{n+1}$ is $dv := \lambda^n dr d \sigma$, where  $d\sigma$ denotes the area element of the unit sphere $\mathbb{S}^n$.
	
	A star-shaped, $C^1$ hypersurface $(\Sigma,g) \subset (M^{n+1}, \bar{g})$ can be written as 
	\begin{equation*}
	\Sigma = \lbrace ( r(\theta), \theta ) \left\vert \theta \in \mathbb{S}^n  \right.\rbrace,
	\end{equation*}
	where $r(\theta)$ is a $C^1$ function on $\mathbb{S}^n$ and $a \leq r(\theta) \leq b$ for all $\theta \in \mathbb{S}^n$.
	We define $\Omega$ as the domain bounded by $\Sigma$ and the slice $\{a\} \times \mathbb{S}^n$ in $M^{n+1}$.
	For a function $f(\theta)$ defined on $\mathbb{S}^n$, we use $D f$ to denote the gradient of $f$.
	Let $V = \lambda(r) \partial_r$ be a conformal Killing vector field in $M^{n+1}$. The support function of $\Sigma$ is given by $u := \metric{V}{\nu}$, where $\nu$ is the unit outer normal of $\Sigma$. Let $\sigma_{ij}$ denote the standard metric on $\mathbb{S}^n$.
	Then we know that (see e.g. \cite{HL21})
	\begin{align}
	g_{ij} =& \lambda^2 \sigma_{ij} + D_i r D_j r, \nonumber\\
	u =& \frac{\lambda^2}{\sqrt{\lambda^2 + \abs{D r}^2}} \label{support function - radial graph}.
	\end{align}
	Hence we get 
	\begin{equation}\label{area-ele-radial graph}
	d\mu = \sqrt{\det \( g_{ij} \)} d\sigma = \lambda^{n-1} \sqrt{\lambda^2 + \abs{D r}^2} d\sigma
	= \frac{\lambda^{n+1}}{u} d\sigma.
	\end{equation}
	We remark that the above formulas are valid for star-shaped hypersurfaces in $\mathbb{R}^{n+1}$ by viewing $\lambda(r) =r$.
	
	In Section \ref{Sec-3}, we define $\pi: M^{n+1} \to \mathbb{R}^{n+1}$ by $\pi \( (r, \theta)\) =  \rho \theta$, where $\rho = \lambda(r)$.  We denote $\widehat{\Sigma} = \pi(\Sigma)$ and $\widehat{\Omega} = \pi (\Omega)$ as in Section \ref{Sec-2}. It is easy to see that \eqref{rel-weighted vol-Ecul vol} is valid in this situation.
	\begin{lem}\label{lem-rel-warped-area}
		Let $\Sigma$ be a star-shaped, $C^1$ hypersurface in $M^{n+1}$. Let $\phi$ be any continuous function on $\mathbb{R}$.
		Then 
		\begin{equation}\label{rel-warped-area}
		\int_\Sigma \phi(\lambda)\lambda' d\mu = \int_{\widehat{\Sigma}} \phi(\rho) \sqrt{\frac{\Psi^2(\rho)-1}{\rho^2} \widehat{u}^2+1}   d\widehat{\mu},
		\end{equation}
		where $\widehat{u}$ denotes the support function of $\widehat{\Sigma} \subset \mathbb{R}^{n+1}$ and $\Psi (t) := \lambda' \circ \lambda^{-1} (t)$.
	\end{lem}
	\begin{proof}
		Since $\rho= \lambda(r)$, formula \eqref{support function - radial graph} yields that
		\begin{equation}\label{supp-Sigmahat}
		\widehat{u} = \frac{\rho^2}{\sqrt{\rho^2 + \abs{D \rho}^2}} = \frac{\lambda^2}{\sqrt{\lambda^2 + \(\lambda'\)^2 \abs{D r}^2}}, 
		\end{equation}	
		By a direct calculation, \eqref{supp-Sigmahat} deduces that
		\begin{equation*}
		\abs{D r}^2 = \frac{\lambda^2 \(\lambda^2 - \widehat{u}^2\)}{\(\lambda'\)^2 \widehat{u}^2}.
		\end{equation*}
		Inserting this formula into \eqref{support function - radial graph}, we obtain
		\begin{equation}\label{rel-u-uhat}
		u = \frac{\lambda^2}{ \sqrt{\lambda^2 + \frac{\lambda^2 (\lambda^2 - \widehat{u}^2)}{\(\lambda'\)^2 \widehat{u}^2}}}= \frac{\lambda' \lambda  \widehat{u}}{\sqrt{\lambda^2 +\( \(\lambda'\)^2-1 \) \widehat{u}^2}}.
		\end{equation}
		Now we focus on the relationship between area elements of $\Sigma$ and $\widehat{\Sigma}$. Formula \eqref{area-ele-radial graph} implies that
		\begin{equation}\label{area ele-Sigma-Sigmahat}
		d\mu = \frac{\lambda^{n+1}}{u} d\sigma, \quad d \widehat{\mu} = \frac{\rho^{n+1}}{\widehat{u}} d\sigma.
		\end{equation}
		Combining \eqref{area ele-Sigma-Sigmahat} with \eqref{rel-u-uhat}, we have
		\begin{equation*}
		\phi(\lambda)\lambda' d\mu = \frac{\phi(\lambda)\lambda' \lambda^{n+1} }{u} d\sigma
		= \frac{\phi(\rho)\lambda' \widehat{u}}{u} d\widehat{\mu}
		= \phi(\rho) \sqrt{\frac{\Psi^2(\rho)-1}{\rho^2} \widehat{u}^2+1} d\widehat{\mu}.
		\end{equation*}
		Then Lemma \ref{lem-rel-warped-area} follows by integrating both  sides of the above formula on $\Sigma$ and $\widehat{\Sigma}$ respectively.
	\end{proof}
	
	Now we use Theorem C to derive the following weighted isoperimetric inequalities for domains in an annulus in $\mathbb{R}^{n+1}$.
	\begin{lem}\label{lem-ThmC-modified}
		Let $\widehat{\Sigma} \subset \mathbb{R}^{n+1}$ be a star-shaped hypersurface in the domain bounded by the geodesic spheres $\{\lambda(a)\} \times \mathbb{S}^n$ and $\{\lambda(b)\} \times \mathbb{S}^n$. Let $\widehat{\Omega}$ be the domain bounded by $\widehat{\Sigma}$ and $\{\lambda(a)\} \times \mathbb{S}^n$.  Let $\phi(t)$ be a smooth, positive, even function on $[-\lambda(b), \lambda(b)]$ that satisfies $\(\log \phi\)'' \geq  0$. Then
		\begin{equation*}
		\int_{\widehat{\Sigma}} \phi(\rho) d\widehat{\mu} \geq \tilde{\xi} \( \int_{\widehat{\Omega}} \phi(\rho) d\widehat{v}\),
		\end{equation*}
		where $\rho$ denotes the radial function of $\widehat{\Sigma}$, and $\tilde{\xi}$ satisfies
		\begin{equation*}
		\tilde{\xi} \( \omega_n \int_{\lambda(a)}^t \phi(s) s^n ds\) = \omega_n \phi(t) t^n  , \quad \forall \ t \in [\lambda(a),\lambda(b)].
		\end{equation*}
	\end{lem}
	
	\begin{proof}
		Theorem C implies that
		\begin{equation*}
		\begin{aligned}
		\int_{\widehat{\Sigma}}\phi(\rho) d\widehat{\mu} \geq& \xi \(\int_{\widehat{\Omega} \cup B( \lambda(a))  } \phi(\rho) d\widehat{v} \)\\
		=&\xi \(\int_{\widehat{\Omega}} \phi(\rho) d\widehat{v}+  \omega_n \int_0^{\lambda(a)} \phi(s)s^n ds    \)\\
		=&\tilde{\xi} \( \int_{\widehat{\Omega}} \phi(\rho) d\widehat{v}\),
		\end{aligned}
		\end{equation*}
		where we used \eqref{def-xi}.
		We complete the proof of Lemma \ref{lem-ThmC-modified}.
	\end{proof}
	Next, we prove the weighted isoperimetric inequalities in warped product manifolds.
	\begin{proof}[Proof of Theorem \ref{thm-iso-warped}]
		Jensen's inequality yields that
		\begin{equation}\label{Jensen-general}
		\int_{\widehat{\Sigma}} \phi(\rho) \sqrt{\(\Lambda(\rho) \widehat{u}\)^2+1} d\widehat{\mu}
		\geq \sqrt{ \( \int_{\widehat{\Sigma}} \phi(\rho) \Lambda(\rho) \widehat{u} d\widehat{\mu}\)^2 + \(\int_{\widehat{\Sigma}} \phi(\rho) d\widehat{\mu} \)^2 },
		\end{equation} 
		with equality if and only if $\Lambda(\rho) \widehat{u}$ is constant on $\widehat{\Sigma}$.
		By the divergence theorem, we have
		\begin{equation}\label{div-formula}
		\begin{aligned}
		&\int_{\widehat{\Sigma}} \phi(\rho) \Lambda(\rho) \widehat{u} d\widehat{\mu}
		- \omega_n \phi(\lambda(a)) \Lambda(\lambda(a)) \lambda^{n+1}(a)\\
		=& (n+1) \int_{\widehat{\Omega}} \phi(\rho) \Lambda(\rho) d\widehat{v} + \int_{\widehat{\Omega}} 
		\( \phi(\rho) \Lambda(\rho ) \)' \rho d\widehat{v} .
		\end{aligned}
		\end{equation}
		The assumptions of $\Lambda(t)$ and $\phi(t)$ imply that 
		\begin{equation*}
		\begin{aligned}
		\( \frac{\phi(\rho) \Lambda(\rho)}{\phi(\rho)} \)' \geq& 0,\\
		\(\frac{\( \phi(\rho) \Lambda(\rho ) \)' \rho}{\phi(\rho)}\)' 
		=& \(\log \phi\)'' \lambda \rho + \(\log \phi \)' \(\Lambda \rho \)' + \(\Lambda' \rho\)' 
		\geq 0.
		\end{aligned}
		\end{equation*}
		Then we can apply the symmetrization method as in Lemma \ref{lem-sym of phi} to get
		\begin{equation}\label{sym-general}
		\int_{\widehat{\Sigma}} \phi(\rho) \Lambda(\rho) \widehat{u} d\widehat{\mu}
		\geq  \tilde{\eta} \( \int_{\widehat{\Omega}} \phi(\rho) d\widehat{v} \),
		\end{equation}
		where $\tilde{\eta}$ is a single-variable, increasing function such that equality holds in \eqref{sym-general} when $\widehat{\Sigma} \subset \mathbb{R}^{n+1}$ is a geodesic sphere centered at the origin. Meanwhile, Lemma \ref{lem-ThmC-modified} asserts that
		\begin{equation}\label{iso-general}
		\int_{\widehat{\Sigma}} \phi(\rho) d \widehat{\mu} \geq \tilde{\xi} \( \int_{\widehat{\Omega}} \phi(\rho) d\widehat{v}\).
		\end{equation}
		Using \eqref{sym-general} and \eqref{iso-general} in \eqref{Jensen-general}, we have
		\begin{equation}\label{final result on R^n+1}
		\int_{\widehat{\Sigma}} \phi(\rho) \sqrt{\(\Lambda(\rho) \widehat{u}\)^2+1} d\widehat{\mu}
		\geq \tilde{\psi} \(  \int_{\widehat{\Omega}} \phi (\rho) d\widehat{v} \).
		\end{equation}
		If equality holds in \eqref{final result on R^n+1}, then $\Lambda(\rho) \widehat{u}$ is constant.  Let $\rho_{\max} = \max_{\theta \in \mathbb{S}^n}  \rho(\theta)$ and $\rho_{\min} = \min_{\theta \in \mathbb{S}^n}  \rho(\theta)$. Then \eqref{supp-Sigmahat} implies that $\Lambda (\rho) \widehat{u} = \Lambda(\rho_{\max}) \rho_{\max} =  \Lambda(\rho_{\min}) \rho_{\min}$. Since $\Lambda' \geq 0$, we know that $\rho$ is constant on $\widehat{\Sigma}$ in this case, which implies that $\widehat{\Sigma}$ is a geodesic sphere centered at the origin.  Finally, using \eqref{rel-weighted vol-Ecul vol} and \eqref{rel-warped-area} in \eqref{final result on R^n+1}, we obtain \eqref{weighted isope warped produc}. Besides, equality holds in \eqref{weighted isope warped produc} if and only if $\Sigma$ is a radial coordinate slice $\{r_0\} \times \mathbb{S}^n$. We finish the proof of Theorem \ref{thm-iso-warped}.
	\end{proof}
	
	Finally, we prove a weighted isoperimetric inequality in the anti-de Sitter-Schwarzschild manifold.
	\begin{proof}[Proof of Theorem \ref{thm-iso-ads-S}]
		Here we adapt the notations in the proof of Theorem \ref{thm-iso-warped} by viewing $\phi(\rho) =1$.
		By the property of $\lambda(r)$ in \eqref{warp-func-ads-S}, we have for all $r \geq a$,
		\begin{equation}\label{lambda-prop}
		\begin{aligned}
		\(\lambda'(r) \)^2-1 =& \lambda^2(r) \(1- m \lambda^{-n-1} (r)\) \geq 0, \\
		\lambda''(r) \lambda(r) - \( \lambda'(r)\)^2+1 =& \frac{m(n+1)}{2} \lambda^{1-n}(r)>0. 
		\end{aligned}
		\end{equation}
		Hence $\Lambda(t)$ is well defined and  $\Lambda (\lambda(a)) =0$. 
		Then we use $\rho = \lambda(r)$ and \eqref{lambda-prop} to get
		\begin{equation}\label{Lambda'(rho)}
		\begin{aligned}
		\Lambda' (\rho) =& \frac{d \Lambda (\lambda(r))}{dr} \frac{dr}{d\rho}\\
		=&\frac{d}{dr}\(\frac{\sqrt{\(\lambda' \)^2-1}}{\lambda} \) \frac{1}{\lambda'}\\
		=&\frac{\lambda'' \lambda- \( \lambda'\)^2 +1}{\lambda^2 \sqrt{\(\lambda'\)^2-1 } }  \\
		=& \frac{m(n+1)}{2} \frac{1}{\lambda^{n+1} \sqrt{\(\lambda'\)^2-1 }}\\
		=& \frac{m(n+1)}{2} \frac{1}{\rho^{n+2} \Lambda(\rho) } 
		\geq 0.  
		\end{aligned}
		\end{equation}
		Therefore, by a similar symmetrization argument as in Lemma \ref{lem-sym of phi}, we have
		\begin{equation}\label{sym-ads-s}
		\int_{\widehat{\Omega}} \Lambda(\rho) d\widehat{v} \geq 
		\widehat{\eta}\( \int_{\widehat{\Omega}} d\widehat{v}  \),
		\end{equation}
		where $\widehat{\eta}$ was defined in \eqref{def-eta-hat}.
		Equality holds in \eqref{sym-ads-s} if $\widehat{\Sigma}$ is a geodesic sphere of radius greater than $m^{\frac{1}{n+1}}$ centered at the origin in $\mathbb{R}^{n+1}$. 
		From \eqref{Lambda'(rho)}, we have
		\begin{equation}\label{integ-Lam'rho}
		\int_{\widehat{\Omega}} \Lambda'(\rho) \rho d\widehat{v} = \frac{m(n+1)}{2}\int_{\widehat{\Omega}} \frac{1}{\rho^{n+1}  \Lambda(\rho)} d\widehat{v}. 
		\end{equation}
		Inserting \eqref{sym-ads-s} and \eqref{integ-Lam'rho} into \eqref{div-formula}, we obtain
		\begin{equation}\label{term1 div ads-S}
		\int_{\widehat{\Sigma}} \Lambda(\rho) \widehat{u} d\widehat{\mu } 
		\geq  \(n+1\)\widehat{\eta} \( \int_{\widehat{\Omega}} d\widehat{v}  \)+ \frac{m(n+1)}{2}\int_{\widehat{\Omega}} \frac{1}{\rho^{n+1}  \Lambda(\rho)} d\widehat{v}.
		\end{equation}
		The classical isoperimetric inequality asserts
		\begin{equation}\label{class-iso-ads-S}
		\int_{\widehat{\Sigma}} d\widehat{\mu} \geq \omega_n^{\frac{1}{n+1}} \( (n+1) \int_{\widehat{\Omega}} d\widehat{v} +\omega_n m  \)^{\frac{n}{n+1}}.
		\end{equation}
		Here we remark that $\frac{\omega_n m}{n+1}$ is the volume of $B(m^{\frac{1}{n+1}}) \subset \mathbb{R}^{n+1}$. Using \eqref{term1 div ads-S} and \eqref{class-iso-ads-S} in \eqref{Jensen-general}, we know
		\begin{equation*}
		\begin{aligned}
		\(\int_{\widehat{\Sigma}} \sqrt{\(\Lambda(\rho) \widehat{u}\)^2+1 } d\widehat{\mu}\)^2
		\geq&
		\( \(n+1\)\widehat{\eta} \( \int_{\widehat{\Omega}} d\widehat{v}  \)+ \frac{m(n+1)}{2}\int_{\widehat{\Omega}} \frac{1}{\rho^{n+1}  \Lambda(\rho)} d\widehat{v}  \)^2\\
		&+  \( (n+1)\omega_n^{\frac{1}{n}} \int_{\widehat{\Omega}} d\widehat{v} +\omega_n^{\frac{n+1}{n}} m  \)^\frac{2n}{n+1}  .
		\end{aligned}
		\end{equation*}
		Then we can apply  \eqref{rel-warped-area} and \eqref{rel-weighted vol-Ecul vol} to get
		\begin{equation*}
		\begin{aligned}
		\(\int_\Sigma \lambda' d\mu\)^2 
		\geq& 
		\(\(n+1\)\widehat{\eta}\(\int_\Omega \lambda' dv  \)
		+ \frac{m(n+1)}{2}\int_{\Omega} \frac{\lambda'}{\lambda^n \sqrt{\( \lambda' \)^2-1}}dv\)^2 \\
		&+ \( (n+1) \omega_n^{\frac{1}{n}}  \int_\Omega \lambda' dv+ \omega_n^{\frac{n+1}{n}} m  \)^\frac{2n}{n+1}.
		\end{aligned}
		\end{equation*}
		Thus we proved \eqref{weighted isope adsS}.
		By \eqref{Lambda'(rho)} and the same argument as in the proof of Theorem \ref{thm-iso-warped}, we know that equality holds in \eqref{weighted isope adsS} if and only if $\Sigma$ is a radial coordinate slice $\{r_0\} \times \mathbb{S}^n$, where $r_0 \in [a, +\infty)$. We complete the proof of Theorem \ref{thm-iso-ads-S}.
	\end{proof}

\end{document}